\documentclass{amsart}
\usepackage{amsfonts}
\usepackage[all]{xy}
\usepackage{amssymb,amsmath}
\newtheorem{thm}{Theorem}[section]
\newtheorem{cor}[thm]{Corollary}

\newtheorem{defn}[thm]{Definition}

\newtheorem{exm}[thm]{Example}

\numberwithin{equation}{section}

\begin{document}
\title{ weakly $(m,n)$-closed ideals and $(m,n)$-von Neumann regular rings}
\author{David F. Anderson, Ayman Badawi, and  Brahim Fahid}
\address{Depaertment of Mathematics, The University of Tennessee, Knoxville, TN 37996-1320, U.S.A.}\email{anderson@math.utk.edu}
\address{Department of Mathematics  $\&$ Statistics, The American University of
Sharjah, P.O.  Box 26666, Sharjah, United Arab Emirates}\email{abadawi@aus.edu}
\address{Department of Mathematics, Faculty of Sciences, B.P. 1014, Mohammed V University, Rabat, Morocco} \email{fahid.brahim@yahoo.fr}
\subjclass[2010]{Primary 13A15; Secondary 13F05, 13G05.}
\keywords{prime ideal, radical ideal, $2$-absorbing ideal, $n$-absorbing ideal,
$(m,n)$-closed ideal, weakly $(m,n)$-closed ideal, $(m,n)$-von Neuman regular.}

\begin{abstract}
Let $R$ be a commutative ring with $ 1 \neq 0$, $I$  a proper ideal of $R$, and $m$ and $n$ positive integers.
In this paper,  we define $I$ to be a {\it weakly $(m,n)$-closed ideal}
if $ 0\neq x^{m}\in I$ for $x \in R$ implies $x^{n} \in I$,
and $R$ to be an {\it $(m,n)$-von Neumann regular ring} if for every $x \in R$,
there is an $r \in R$ such that $x^mr = x^n$.
A number of results concerning weakly $(m, n)$-closed ideals and $(m,n)$-von Neumann regular rings are given.
\end{abstract}

\date{\today}
 \maketitle

\section{Introduction}  \label{sec1}

Let $R$ be a commutative ring with $1 \neq 0$, $I$ a proper ideal of $R$, and $n$ a positive integer.
As in \cite{BA}, $I$ is an {\it $n$-absorbing} (resp., {\it strongly $n$-absorbing}) {\it ideal} of $R$
if whenever $x_1 \cdots x_{n+1} \in  I$ for $x_1, \ldots, x_{n+1} \in R$
(resp., $I_1 \cdots I_{n+1} \subseteq I$ for ideals $ I_1, \ldots, I_{n+1}$ of $R$),
then there are $n$ of the $x_i$'s (resp., $n$ of the $I_i$'s) whose product is in $I$.
As in \cite{BARECENT}, $I$ is a {\it semi-$n$-absorbing ideal} of $R$ if
$x^{n+1} \in I$  for $x \in R$ implies $x^n \in I$; and
for positive integers $m$ and $n$, $I$ is an {\it $(m, n)$-closed ideal} of $R$
if $ x^m \in I$ for $x \in R$ implies $x^n \in I$.
And, as in \cite{MSD},
$I$ is a {\it weakly $n$-absorbing}
(resp., {\it strongly weakly $n$-absorbing}) $ideal$ of $R$ if whenever
$0 \neq x_1 \cdots x_{n+1} \in  I$ for $x_1, \ldots , x_{n+1} \in R$
(resp., $0 \neq I_1 \cdots I_{n+1} \subseteq I$ for ideals $ I_1, \ldots, I_{n+1}$ of $R$),
then there are $n$ of the $x_i$'s (resp., $n$ of the $I_i$'s) whose product is in $I$.

In this paper, we define $I$ to be a {\it weakly semi-$n$-absorbing ideal} of $R$
if $0 \neq x^{n+1} \in I$ for $x \in R$ implies $x^n \in I$. More generally,
for positive integers $m$ and $n$, we  define $I$ to be a {\it weakly $(m,n)$-closed ideal} of $R$
if $ 0\neq x^{m}\in I$ for $x \in R$ implies $x^{n} \in I$. Thus $I$ is a weakly
semi-$n$-absorbing ideal if and only if $I$ is a weakly $(n+1,n)$-closed ideal.
Moreover, an $(m,n)$-closed ideal is a weakly
$(m,n)$-closed ideal, and the two concepts agree when $R$ is reduced.
Every proper ideal is weakly $(m,n)$-closed for $m \leq n$;
so we usually assume that $m > n$.

The above definitions all concern generalizations of prime ideals.
A $1$-absorbing ideal is just a prime ideal,
and a weakly $1$-absorbing ideal is just a weakly prime ideal
(a proper ideal $I$ of $R$ is a {\it weakly prime ideal} if
$0 \neq xy \in I$ for $x, y \in R$ implies $x \in I$ or $y \in I$).
A proper ideal is a radical ideal if and only if it is $(2,1)$-closed.
However, a weakly $(2,1)$-closed ideal need not be a weakly radical ideal
(a proper ideal $I$ of $R$ is a {\it weakly radical ideal} if $0 \neq x^n \in I$ for $x \in R$ and
$n$ a positive integer implies $x \in I$) (see Example~\ref{ex1}(b)).

Weakly prime ideals and weakly radical ideals were studied in \cite{SA},
and weakly radical (semiprime) ideals have been studied in more detail in \cite{BSEMI}.
The concept of $2$-absorbing ideals was introduced in \cite{B}
and then extended to $n$-absorbing ideals in \cite{BA}.
Related concepts include $2$-absorbing primary ideals (see \cite{YTB1}),
weakly $2$-absorbing ideals (see \cite{BD}), weakly 2-absorbing primary ideals
(see \cite{YTB}),  and $(m, n)$-closed ideals (see \cite{BARECENT}).
Other generalizations and related concepts are investigated
in \cite{SA}, \cite{EF}, \cite{BSEMI}, \cite{BD}, \cite{FD1}, \cite{FD}, and \cite{MSD}.
For a survey on $n$-absorbing ideals, see \cite{BD1}.

Let $R$ be a commutative ring and $m$ and $n$ positive integers.
We define $R$ to be an {\it $(m,n)$-von Neumann regular ring} if for every $x \in R$,
there is an $r \in R$ such that $x^mr = x^n$.
Thus a $(2,1)$-von Neumann regular ring is just a von Neumann regular ring.
In this paper, we study weakly $(m,n)$-closed ideals, $(m,n)$-von Neumann regular rings,
and the connections between the two concepts.

Let $m$ and $n$ be positive integers with $m > n$.
Among the many results in this paper, we show in Theorem~\ref{thm-nil} that
if $I$ is a weakly $(m,n)$-closed, but not $(m,n)$-closed, ideal of $R$,
then $I \subseteq Nil(R)$.
In Theorem~\ref{add2}, we determine when a proper ideal of $R_1 \times R_2$
is weakly $(m,n)$-closed, but not $(m,n)$-closed; and in
Theorem~\ref{add-2}, we investigate when
a proper ideal of $R(+)M$ is weakly $(m,n)$-closed, but not $(m,n)$-closed.
In Section~\ref{sec3}, we introduce and investigate $(m,n)$-von Neumann regular elements and
$(m,n)$-von Neumann regular rings. It is shown in Theorem~\ref{add3} that every proper ideal of $R$ is weakly $(m,n)$-closed
if and only if every non-nilpotent element of $R$ is $(m,n)$-von Neumann regular
and $w^m = 0$ for every $w \in Nil(R)$. In Theorem~\ref{add4},
we show that every proper ideal of $R$ is $(m,n)$-closed if and only if $R$ is $(m,n)$-von Neumann regular.
Finally, we define the concepts of $n$-regular and $\omega$-regular commutative rings as a way
to measure how far a zero-dimensional commutative ring is from being von Neumann regular.

We assume throughout this paper that all rings are commutative with $ 1\neq 0$, all $R$-modules are unitary,
and $f(1) = 1$ for all ring homomorphisms $f : R \longrightarrow T$.
For such a ring $R$, let $Nil(R)$ be its ideal of nilpotent elements,
$Z(R)$ its set of zero-divisors, $U(R)$ its group of units, char$(R)$ its characteristic, and dim$(R)$ its (Krull) dimension.
Then $R$ is {\it reduced} if $Nil(R) = \{0\}$ and $R$ is {\it quasilocal} if it has exactly one maximal ideal.
As usual, $\mathbb{N}$, $\mathbb{Z}$, and $\mathbb{Z}_n$ will denote the positive integers, integers,
and integers modulo $n$, respectively. Several of our results use the $R(+)M$ construction
as in \cite{HJ}. Let $R$ be a commutative ring and $M$ an $R$-module.
Then $R(+)M = R \times M$ is a commutative ring with identity $(1, 0)$ under addition defined by
$(r,m) +(s, n) = (r + s,m + n)$ and multiplication defined by
$(r,m)(s, n) = (rs, rn + sm)$. Note that $(\{0\}(+)M)^{2} = \{0\}$; so $\{0\}(+)M \subseteq Nil(R(+)M)$.
	
\bigskip

\section{Properties of weakly $(m,n)$-closed ideals} \label{sec2}

In this section, we  give some basic properties of weakly $(m,n)$-closed ideals
and investigate weakly $(m,n)$-closed ideals in several classes of commutative rings.
We start by recalling the definitions of weakly semi-$n$-absorbing and weakly $(m,n)$-closed ideals.

\begin{defn} \label{def-princ}
Let $R$ be a commutative ring, $I$ a proper ideal of $R$, and $m$ and $n$ positive integers.

\begin{enumerate}
  \item  $I$ is a {\it weakly semi-$n$-absorbing ideal}
	of $R$ if $0\neq x^{n+1} \in I$ for $x \in R$ implies $x^n \in I$.
  \item  $I$ is a {\it weakly $(m,n)$-closed ideal} of $R$
	if $0 \neq x^m \in I$ for $x \in R$ implies $x^n \in I$.
\end{enumerate}

\end{defn}

The proof of the next result follows easily from the definitions, and thus will be omitted.

\begin{thm} \label{basic}
Let $R$ be a commutative ring and $m$ and $n$ positive integers.

  \begin{enumerate}
    \item  If $I$ is a weakly $n$-absorbing ideal of $R$,
		then $I$ is weakly semi-$n$-absorbing (i.e., weakly $(n+1,n)$-closed).
    \item If $I$ is a weakly $(m, n)$-closed ideal of $R$,
		then $I$ is weakly $(m, n')$-closed
		for every positive integer $n' \geq n$.
    \item If $I$ is a weakly $n$-absorbing ideal of $R$,
		then $I$ is weakly $(m, n)$-closed for every positive integer $m$.
		\item  An intersection of weakly $(m,n)$-closed ideals of $R$ is weakly $(m,n)$-closed.
  \end{enumerate}
\end{thm}

While an $(m,n)$-closed ideal is always weakly $(m,n)$-closed, the converse need not hold.
If an ideal is $(m,n)$-closed, then it is also $(m',n')$-closed
for all positive integers $m' \leq m$ and $n' \geq n$ \cite[Theorem 2.1(3)]{BARECENT}.
However, a weakly $(m,n)$-closed ideal need not be weakly $(m',n)$-closed for $m' < m$.
We next give two examples to illustrate these differences.

\begin{exm} \label{ex1}
$(a)$ Let $R = \mathbb{Z}_8$ and  $I = \{0,4 \}$.
Then $I$ is weakly $(3,1)$-closed since $x^3 = 0$ for every nonunit $x$ in $R$.
However, $I$ is not $(3,1)$-closed since  $2^3 = 0 \in I$ and $2 \notin I$, and
$I$ is not weakly $(2,1)$-closed since $0 \neq 2^2 = 4 \in I$ and $2 \notin I$.

$(b)$   Let $R = \mathbb{Z}_{16}$ and $I = \{0, 8\}$.
Then $I$ is weakly $(2,1)$-closed since $8$ is not a square in $\mathbb{Z}_{16}$.
However, $I$ is not $(2,1)$-closed since $4^2 = 0 \in I$ and $4 \notin I$,
and $I$ is not a weakly radical ideal (and thus not weakly prime) since $0 \neq 2^3 = 8 \in I$ and $2 \notin I$.
\end{exm}

The following definition will be useful for studying weakly $(m,n)$-closed ideals
that are not ($m,n)$-closed (cf. \cite[Definition 2.2]{BSEMI}).

\begin{defn}
Let $R$ be a commutative ring, $m$ and $n$ positive integers,
and $I$ a weakly $(m,n)$-closed ideal of $R$.
Then $a \in R$ is an $(m,n)$-{\it unbreakable-zero element}
of $I$ if $a^{m} = 0$ and $a^{n} \notin I$. (Thus $I$ has an $(m,n)$-{\it unbreakable-zero element
if and only if $I$ is not $(m,n)$-closed}.)
\end{defn}

\begin{thm} \label{thm-1} $($cf. \cite[Theorem 2.3]{BSEMI}$)$
Let $R$ be a commutative ring, $m$ and $n$ positive integers,
and $I$ a weakly $(m,n)$-closed ideal of $R$.
If $a$ is an $(m,n)$-unbreakable-zero element of $I$,
then $(a + i )^m = 0$ for every $i \in I$.
\end{thm}

\begin{proof}
Let $i \in I$. Then $(a + i )^m = a^m  +\sum_{k=1}^{m}\binom{m}{k}a^{m-k}i^k =
0 + \sum_{k=1}^{m}\binom{m}{k}a^{m-k}i^{k} \in I$, and similarly,
$(a+i)^n \notin I$ since $a^n \notin I$. Thus $(a + i )^m = 0$ since $I$ is weakly $(m,n)$-closed.
\end{proof}

\begin{thm} $($cf. \cite[p. 839]{SA} and \cite[Theorems 2.4 and 2.5]{BSEMI}$)$   \label{thm-nil}
Let $R$ be a commutative ring, $m$ and $n$ positive integers,
and $I$ a weakly $(m,n)$-closed ideal of $R$.
If $I$ is not $(m,n)$-closed, then $I \subseteq Nil(R)$.
Moreover, if $I$ is not $(m,n)$-closed and char$(R) = m$ is prime,
then $i^m = 0$ for every $i \in I$.
\end{thm}

\begin{proof}
Since $I$ is a weakly $(m,n)$-closed ideal of $R$ that is not $(m,n)$-closed,
$I$ has an $(m,n)$-unbreakable-zero element $a$. Let $i \in I$.
Then $a^m = 0$, and $(a + i)^m = 0$ by Theorem \ref{thm-1};  so $a, a + i \in Nil(R)$.
Thus $i = (a + i) - a \in Nil(R)$;
so $I \subseteq Nil(R)$.

The ``moreover'' statement is clear since $0 = (a + i)^m = a^m + i^m = i^m$
when char$(R) = m$ is prime.
\end{proof}

The next two theorems are the analogs of the results for $(m,n)$-closed ideals
in \cite[Theorem 2.8]{BARECENT} and \cite[Theorem 2.10]{BARECENT}, respectively.
Their proofs are similar, and thus will be omitted.

\begin{thm}
Let $R$ be a commutative ring, I a proper ideal of $R$,
$S \subseteq R \setminus \{0\}$ a multiplicative set,
and $m$ and $n$ positive integers.
If $I$ is a weakly $(m,n)$-closed ideal of $R$,
then $I_S$ is a weakly $(m,n)$-closed ideal of $R_S$.
\end{thm}

\begin{thm} \label{extthm}
Let $f : R \longrightarrow T$ be a homomorphism of commutative rings and $m$ and $n$ positive integers.
 \begin{enumerate}
    \item If $f$ is injective and $J$ is a weakly $(m,n)$-closed ideal of $T$,
		then $f^{-1}(J)$ is a weakly $(m,n)$-closed ideal of $R$.
		In particular, if $R$ is a subring of $T$ and $J$ is a weakly $(m,n)$-closed ideal of $T$,
		then $J \cap R$ is a weakly $(m,n)$-closed ideal of $R$.
		
    \item  If $f$ is surjective and $I$ is a weakly $(m,n)$-closed ideal of $R$ containing ker$f$,
		then $f(I)$ is a weakly $(m,n)$-closed ideal of $T$. In particular, if $I$ is a weakly $(m,n)$-closed ideal of $R$ and
		$J \subseteq I$ is an ideal of $R$, then $I/J$ is a weakly $(m,n)$-closed ideal of $R/J$.
    \end{enumerate}
\end{thm}

In the following theorems, we determine when an ideal of
$R_1 \times R_2$ is weakly $(m,n)$-closed, but not $(m,n)$-closed.
(Recall that an ideal of $R_1 \times R_2$ has the form $I_1 \times I_2$
for ideals $I_1$ of $R_1$ and $I_2$ of $R_2$.)
It is easy to determine when an ideal of $R_1 \times R_2$ is $(m,n)$-closed.

\begin{thm}  \label{addpro} $($cf. \cite[Theorem 2.12]{BARECENT}$)$
Let $R = R_{1} \times R_{2}$, where $R_{1}$ and $R_{2}$ are commutative rings,
$J$  a proper ideal of $R$,
and $m$ and $n$ positive integers.
Then the following statements are equivalent.
\begin{enumerate}
\item  $J$ is an $(m,n)$-closed ideal of $R$.
\item  $J = I_1 \times R_2$, $R_1 \times I_2$, or $I_1 \times I_2$ for
$(m,n)$-closed ideals $I_1$ of $R_1$ and $I_2$ of $R_2$.
\end{enumerate}
\end{thm}

\begin{proof}
This follows directly from the definitions.
\end{proof}

The analog of $(1) \Rightarrow (2)$ of Theorem~\ref{addpro} clearly holds
for weakly $(m,n)$-closed ideals by Theorem~\ref{extthm}(2),
but our next theorem shows that the analog of $(2) \Rightarrow (1)$
does not hold for weakly $(m,n)$-closed ideals.

\begin{thm} \label{add1}
Let $R = R_{1} \times R_{2}$, where $R_{1}$ and $R_{2}$ are commutative rings,
$I_1$ a proper ideal of $R_1$,
and $m$ and $n$ positive integers.
Then the following statements are equivalent.
  \begin{enumerate}
    \item $I_1\times R_2$ is a weakly $(m,n)$-closed ideal of $R$.
    \item  $I_1$ is an $(m, n)$-closed ideal of $R_1$.
    \item $I_1 \times R_2$ is an $(m, n)$-closed ideal of $R$.
  \end{enumerate}
	A similar result holds for $R_1 \times I_2$ when $I_2$ is a proper ideal of $R_2$.
\end{thm}

\begin{proof}
$(1) \Rightarrow (2)$  $I_1$ is a weakly $(m, n)$-closed ideal of $R_1$ by Theorem~\ref{extthm}(2).
If $I_1$ is not an $(m, n)$-closed  ideal of $R_1$,
then $I_1$ has an $(m, n)$-unbreakable-zero element $a$.
Thus $(0, 0) \not = (a, 1)^m \in I_1 \times R_2$, but $(a, 1)^n \not \in I_1 \times R_2$, a contradiction.
Hence $I_1$ is an $(m, n)$-closed ideal of $R_1$.

$(2) \Rightarrow (3)$ This is clear (cf.
\cite[Theorem 2.12]{BARECENT}).

$(3) \Rightarrow (1)$ This is clear by definition.
\end{proof}

\begin{thm} \label{add2}
Let $R = R_{1} \times R_{2}$, where $R_{1}$ and $R_{2}$ are commutative rings,
$J$  a proper ideal of $R$, and $m$ and $n$ positive integers.
Then the following statements are equivalent.
  \begin{enumerate}
    \item $J$ is a weakly $(m,n)$-closed ideal of $R$ that is not $(m,n)$-closed.
    \item  $J = I_{1} \times I_{2}$ for proper ideals $I_1$ of $R_1$ and $I_2$ of $R_2$ such that either
		\begin{enumerate}
		\item  $I_{1}$ is a weakly $(m, n)$-closed ideal
		of $R_1$ that is not $(m, n)$-closed, $y^m = 0$ whenever $y^m \in I_2$ for $y \in R_2$
		(in particular, $i^m = 0$ for every $i \in I_2$), and if $0 \not = x^m \in I_1$ for some $x \in R_1$,
		then $I_2$ is an $(m, n)$-closed ideal of $R_2$, or
		\item $I_{2}$ is a weakly $(m, n)$-closed ideal of $R_2$
		that is not $(m, n)$-closed, $y^m = 0$
		whenever $y^m \in I_1$ for $y \in R_1$ (in particular, $i^m = 0$ for every $i \in I_1$),
		and if $0 \not = x^m \in I_2$ for some $x \in R_2$, then $I_1$ is an $(m, n)$-closed ideal of $R_1$.
		\end{enumerate}
  \end{enumerate}
\end{thm}

\begin{proof}
$(1) \Rightarrow (2)$ Since $J$ is not an $(m, n)$-closed ideal of $R$,
by Theorem \ref{add1} we have $J = I_1\times I_2$, where $I_1$ is a proper ideal of $R_1$
and $I_2$ is a proper ideal of $R_2$. Since $J$ is not an $(m, n)$-closed ideal of $R$,
either $I_1$ is a weakly $(m, n)$-closed ideal of $R_1$  that is not
$(m, n)$-closed or $I_2$ is a weakly  $(m, n)$-closed ideal of $R_2$
that is not $(m, n)$-closed. Assume that $I_1$ is a weakly $(m, n)$-closed ideal of $R_1$
that is not $(m, n)$-closed. Thus $I_1$ has an $(m, n)$-unbreakable-zero element $a$.
Assume that $y^m \in I_2$ for $y \in R_2$. Since $a$ is an $(m, n)$-unbreakable-zero element
of $I_1$ and $(a, y)^m \in J$, we have $(a, y)^m = (0, 0)$.
Hence $y^m = 0$ (in particular, $i^m = 0$ for every $i \in I_2$).
Now assume that $0 \not = x^m \in I_1$ for some $x \in R_1$.
Let $y \in R_2$ such that $y^m \in I_2$. Then $(0, 0) \not = (x, y)^m \in J$.
Thus $y^n \in I_2$, and hence $I_2$ is an $(m, n)$-closed ideal of $R_2$.
Similarly, if $I_2$ is a weakly $(m, n)$-closed ideal of $R_2$ that is not $(m, n)$-closed,
then $y^m = 0$ whenever $y^m \in I_1$ for $y \in R_1$ (in particular, $i^m = 0$ for every $i \in I_1$),
and if $0 \not = x^m \in I_2$ for some $x \in R_2$, then $I_1$ is an $(m, n)$-closed ideal of $R_1$

$(2) \Rightarrow (1)$ Suppose that $I_{1}$ is a weakly $(m, n)$-closed proper ideal of $R_1$
that is not $(m, n)$-closed, $y^m = 0$ whenever $y^m \in I_2$  for $y \in R_2$
(in particular, $i^m = 0$ for every $i \in I_2$), and if $0 \not = x^m \in I_1$
for some $x \in R_1$, then $I_2$ is an $(m, n)$-closed ideal of $R_2$.
Let $a$ be an $(m, n)$-unbreakable-zero element of $I_1$. Then $(a, 0)$ is an
$(m, n)$-unbreakable-zero element of $J$. Thus $J$ is not an $(m, n)$-closed ideal of $R$.
Now assume that $(0, 0) \not = (x, y)^m  = (x^m, y^m) \in J$ for $x \in R_1$ and $y \in R_2$.
Then $(0, 0) \not = (x, y)^m = (x^m, 0) \in J$ and $0 \not = x^m \in I_1$.
Since $I_1$ a weakly $(m, n)$-closed ideal of $R_1$ and $I_2$ is an
$(m, n)$-closed ideal of $R_2$, we have $(x, y)^n \in J$.
Similarly, assume that $I_{2}$ is a weakly $(m, n)$-closed ideal of $R_2$
that is not $(m, n)$-closed, $y^m = 0$ whenever $y^m \in I_1$ for $y \in R_1$
(in particular, $i^m = 0$ for every $i \in I_1$),  and if $0 \not = x^m \in I_2$
for some $x \in R_2$, then $I_1$ is an $(m, n)$-closed ideal of $R_1$.
Then again, $J$ is a weakly $(m,n)$-closed ideal of $R$ that is not $(m,n)$-closed.
\end{proof}

We next consider when certain ideals of $R(+)M$ are weakly $(m,n)$-closed.



\begin{thm} \label{add-2}
Let $R$ be a commutative ring, $I$ a proper ideal of $R$,
$M$ an $R$-module, and $m$ and $n$ postive integers.
Then the following statements are equivalent.
\begin{enumerate}
\item $I(+)M$ is a weakly $(m,n)$-closed ideal of $R(+)M$ that is not $(m, n)$-closed.
\item $I$ is a weakly $(m,n)$-closed ideal of $R$ that is not $(m, n)$-closed
and $m(a^{m-1}M) = 0$ for every $(m, n)$-unbreakable-zero element $a$ of $I$.
\end{enumerate}
\end{thm}

\begin{proof}
$(1) \Rightarrow (2)$  Let $J = I(+)M$. Assume that $0 \not = r^m \in I$ for $r \in R$.
Thus $(0, 0) \not = (r, 0)^m = (r^m, 0) \in J$.
Hence $(r, 0)^n = (r^n, 0) \in J$; so $r^n \in I$.
Thus $I$ is a weakly $(m, n)$-closed ideal of $R$.
Since $J$ is not $(m, n)$-closed, $J$, and hence
$I$, has an $(m, n)$-unbreakable-zero element;
so $I$ is not $(m, n)$-closed.
Let $a$ be an $(m, n)$-unbreakable-zero element of $I$ and $x \in M$.
Then $(a, x)^m = (a^m, m(a^{m-1}x)) \in J$.
Since $a^n \not \in I$, we have $(a, x)^m = (a^m, m(a^{m-1}x)) = (0, 0)$.
Thus $m(a^{m-1}M) = 0$.

$(2) \Rightarrow (1)$  Since $I$ is a weakly
$(m, n)$-closed ideal of $R$ that is not $(m, n)$-closed,
$I$ has an $(m, n)$-unbreakable-zero element $a$. Hence $(a, 0)$ is an
$(m, n)$-unbreakable-zero element of $J = I(+)M$.
Thus $J$ is not an $(m, n)$-closed ideal of $A$. Suppose that
$(0, 0) \not = (r, y)^m = (r^m, m(r^{m-1}y)) \in J$.
Then $r$ is not an $(m, n)$-unbreakable-zero element of $I$ by hypothesis.
Hence $(r^n, n(r^{n-1}y)) = (r, y)^n \in J$; so $J$ is a weakly
$(m, n)$-closed ideal of $A$ that is not $(m, n)$-closed.
\end{proof}



We end this section with another way to construct weakly $(m,n)$-closed ideals that are not $(m,n)$-closed.
See \cite [Theorems 3.1 and 3.8]{BARECENT} for similar results for $(m,n)$-closed ideals.

\begin{thm} \label{addNew}
Let $R$ be an integral domain and $I = p^kR$ a principal ideal of $R$,
where $p$ is a prime element of $R$ and $k$ a positive integer.
Let $m$ be a positive integer such that $m < k$, and write $k = mq + r$
for integers $q, r$, where $q\geq 1$ and  $0\leq r < m$.
Then $J = I/p^cR$ is a weakly $(m, n)$-closed ideal of $R/p^cR$
that is not $(m, n)$-closed for positive integers $n < m$ and $c \geq k + 1$
if and only if $r \not = 0$, $k + 1 \leq c \leq m(q+1)$, and $n(q+1) < k$.
\end{thm}

\begin{proof}
Suppose that $J$ is a weakly $(m, n)$-closed ideal of $R/p^cR$ that is not
$(m, n)$-closed for positive integers $n < m$ and $c \geq k + 1$.
It is clear that $r \not = 0$, for if $r = 0$, then $0 \not = (p^q)^m + p^cR \in J$,
but $(p^q)^n + p^cR \not \in  J$. Since $q + 1$ is the smallest positive integer
such that $(p^{(q+1)})^m + p^cR \in J$ and $J$ is not $(m, n)$ closed,
we have  $0 = (p^{(q+1)})^m + p^cR \in J$ and $(p^{(q+1)})^n + p^cR \not \in  J$.
Thus $n(q+1) < k$ and $ k + 1 \leq c \leq (q+1)m$.

Conversely, assume that $r \not = 0$, $k + 1 \leq c \leq m(q+1)$, and $n(q+1) < k$.
Let $x \in R/p^cR$ such that $x^m \in J$.
Then $x = p^iy + p^cR$ for some $y \in R$ such that $p^{(i + 1)} \nmid y$ in $R$.
Since $x^m = (p^i)^m + p^cR \in J$, we have
$i \geq q + 1$. Thus by hypothesis, $x^m = 0$ in $R/p^cR$.
Since $0 = (p^{(q+1)})^m + p^cR \in J$ and $n(q+1) < k$,
we have $(p^{(q+1)})^n + p^cR \not \in J$. Hence $J$ is not $(m, n)$-closed.
\end{proof}

\begin{exm}
$(a)$  Let $R = \mathbb{Z}$, $I = 2^{12}\mathbb{Z}$, and $J = I/2^{13}\mathbb{Z}$.
Then by Theorem~\ref{addNew}, $J$ is a weakly $(5, 3)$-closed ideal
of $\mathbb{Z}/2^{13}\mathbb{Z}$ that is not $(5, 3)$-closed.

$(b)$ Let $R$, $I$, and $J$ be as in part $(a)$ above.
Then $J(+)J$ is a weakly $(5,3)$-closed ideal of $\mathbb{Z}/2^{13}\mathbb{Z}(+)J$ that is not $(5,3)$-closed by Theorem~\ref{add-2}.
\end{exm}

\bigskip

\section{$(m,n)$-von Neumann regular rings} \label{sec3}

In this section, we introduce the concepts of $(m,n)$-von Neumann regular elements
and $(m,n)$-von Neumann regular rings
and use them to determine when every proper ideal of $R$ is $(m,n)$-closed or weakly $(m,n)$-closed.
We also define the related concepts of $n$-regular and $\omega$-regular commutative rings.
First, we handle the case for ideals contained in $Nil(R)$.

\begin{thm}\label {thm-nilideal}
Let $R$ be a commutative ring and $m$ and $n$ positive integers with $m  > n$.
Then every ideal of $R$ contained in $Nil(R)$ is weakly $(m,n)$-closed
if and only if $w^{m} = 0$ for every $w \in Nil(R)$.
\end{thm}

\begin{proof}
Suppose that every ideal of $R$ contained in $Nil(R)$ is weakly $(m,n)$-closed,
but $w^m \neq 0$ for some $w \in Nil(R)$. Let $J = w^mR \subseteq Nil(R)$.
Then $J$ is weakly $(m,n)$-closed and $0 \neq w^m \in J$; so $w^n \in J$ and $w^n \neq 0$ since $n < m$.
Thus $ w^n = w^ma$ for some $a \in R$, and hence $ w^n(1 - w^{m-n}a) = 0$.
Then $1 - w^{m-n}a \in U(R)$ since $w^{m-n}a \in Nil(R))$; so $w^n = 0$, a contradiction.
Thus $w^m = 0$ for every $w \in Nil(R)$.

Conversely, suppose that $w^m = 0$ for every $ w \in Nil(R)$.
Then every ideal of $R$ contained in $Nil(R)$ is weakly $(m,n)$-closed by definition.
\end{proof}

Recall that $x \in R$ is a {\it von Neumann regular element} of $R$
 if $x^{2}r = x$ for some $r \in R$.
Similarly, $x \in R$ is a {\it $\pi$-regular element} of $R$
if $x^{2n}r = x^n$ for some $r \in R$ and positive integer $n$.
Thus $R$ is a von Neumann regular ring  (resp., $\pi$-regular ring) if and only if
every element of $R$ is von Neumann regular (resp., $\pi$-regular).
It is well known that $R$ is $\pi$-regular (resp., von Neumann regular)
if and only if dim$(R) = 0$ (resp., $R$ is reduced and dim$(R) = 0$) \cite[Theorem 3.1, p. 10]{HJ}.
A ring $R$ is a {\it strongly $\pi$-regular ring}
if there is a positive integer $n$ such that for every $x \in R$, we have $x^{2n}r = x^n$ for some $r \in R$.
For a recent article on von Neumann regular and related elements of a commutative ring, see \cite{DAB}.
These concepts are generalized in the next definition.

\begin{defn}
Let $R$ be a commutative ring and $m$ and $n$ positive integers.
Then $x \in R$ is an {\it  $(m,n)$-von Neumann regular} element of $R$
(or {\it $(m,n)$-vnr} for short) if $x^{m}r = x^{n}$ for some
$r \in R$. If every element of $R$ is $(m, n)$-vnr,
then $R$ is an {\it $(m,n)$-von Neumann regular} ring.
\end{defn}

Thus a commutative ring $R$ is von Neumann regular ring if and only if it is $(2,1)$-von Neumann regular, and
$R$ is strongly $\pi$-regular if and only if it is
$(2n,n)$-von Neumann regular for some positive integer $n$.
The next theorem gives some basic facts about $(m,n)$-vnr elements.

\begin{thm} \label{vnrfacts}
Let $R$ be a commutative ring, $x \in R$, and $m$ and $n$ positive integers.
  \begin{enumerate}
    \item  $x$ is $(m,n)$-vnr for $m \leq n$ (so we usually assume that $m > n$).
		\item  If $x$ is $(m,n)$-vnr, then $x$ is $(m',n')$-vnr for all positive integers $m' \leq m$ and $n' \geq n$.
		\item  If $x \in U(R)$ or $x = 0$, then $x$ is $(m,n)$-vnr for all positive integers $m$ and $n$.
		\item  If $x \in R \setminus (Z(R) \cup U(R))$, then $x$ is $(m,n)$-vnr if and only if  $m \leq n$.
		\item  If $x^n = 0$, then $x$ is $(m,n)$-vnr for every positive integer $m$.
		\item  If $x^k = 0$ and $x^{k-1} \neq 0$ for an integer $k \geq 2$, then $x$ is $(m,n)$-vnr
		if and only if $m \leq n$ or $n \geq k$.
		\item  If $x$ is $(m,n)$-vnr with $m > n$, then $x$ is $(m + 1, n)$-vnr.
		Moreover, in this case, $x$  is $(m',n')$-vnr for all positive integers $m'$ and $n' \geq n$.
	  Thus $R$ is von Neumann regular if and only if
		$R$ is $(m,n)$-von Neumann regular for all positive integers $m$ and $n$.
	\end{enumerate}
\end{thm}

\begin{proof}
The proofs of $(1) - (3)$ and $(5)$ are clear.

$(4)$  By $(1)$, $x$ is $(m,n)$-vnr for $m \leq n$. If $m > n$, then $x^mr = x^n$ for $r \in R$
implies $x^{m-n}r = 1$. Thus $x \in U(R)$, a contradiction.

$(6)$  Suppose that $x^mr = x^n$ for $r \in R$, but $m > n$ and $n < k$.
Then $x^{k-1} = x^n(x^{k-n-1}) = (x^mr)(x^{k-n-1}) = x^k(x^{m-n-1}r) = 0$,
a contradiction. Thus $m \leq n$ or $n \geq k$. The converse is clear.

$(7)$ Let $x$ be $(m,n)$-vnr with $m > n$. Then $x^mr = x^n$ for $r \in R$ implies $x^n = x^mr = x^n(x^{m-n}r) =
 (x^mr)(x^{m-n}r) = x^{m + 1}(x^{m-n-1}r^2)$ with $x^{m-n-1}r^2 \in R$. Thus $x$ is $(m+1,n)$-vnr.
The ``moreover'' statement follows by induction and $(2)$.
\end{proof}

\begin{cor} \label{strongcor}
Let $R$ be a commutative ring and $m$ and $n$ positive integers with $m > n$.
Then $R$ is $(m,n)$-von Neumann regular if and only if $R$ is
$(m',n')$-von Neumann regular for all positive integers $m'$ and $n' \geq n$.
In particular, if $R$ is $(m,n)$-von Neumann regular, then $R$ is strongly $\pi$-regular, and thus dim$(R) = 0$.
\end {cor}

We next determine when every proper ideal of $R$ is weakly $(m,n)$-closed.

\begin{thm}\label{add3}
  Let $R$ be a commutative ring and $m$ and $n$ positive integers with $m > n$.
	Then the following statements are equivalent.
  \begin{enumerate}
    \item Every proper ideal of $R$ is weakly $(m,n)$-closed.
    \item Every non-nilpotent element of $R$ is $(m,n)$-vnr
		and $w^m = 0$ for every $w \in Nil(R)$.
  \end{enumerate}
\end{thm}

\begin{proof}
$(1) \Rightarrow (2)$
Since every ideal of $R$ contained in $Nil(R)$ is weakly $(m,n)$-closed,
$w^{m} = 0$ for every  $w \in Nil(R)$ by Theorem \ref{thm-nilideal}.
Let $x \in R \setminus Nil(R)$. If $x \in U(R)$, then $x$ is $(m,n)$-vnr by Theorem~\ref{vnrfacts}(3).
If $x \notin U(R)$, then $I = x^mR$ is weakly $(m,n)$-closed
and $0 \neq x^{m} \in I$; so $x^{n} \in I$. Thus $ x^{n} = x^{m}r$ for some $r \in R$,
and hence $x$ is $(m,n)$-vnr.

$(2) \Rightarrow (1)$  Let $I$ be a proper ideal of $R$ and $0 \not = x^m \in I$ for
$x \in R$. Then $x \notin Nil(R)$; so $x$ is $(m,n)$-vnr.
Thus $x^mr = x^n$ for some $r \in R$;
so $x^n = x^mr \in I$. Hence $I$ is weakly $(m,n)$-closed.
\end{proof}

In view of Theorem \ref{add3}, we have the following result.

\begin{cor}  \label{vnrcor}
  Let $R$ be a reduced commutative ring and $m$ and $n$ positive integers.
	Then the following statements are equivalent.
  \begin{enumerate}
    \item Every proper ideal of $R$ is weakly $(m,n)$-closed.
    \item Every proper ideal of $R$ is $(m,n)$-closed.
    \item $R$ is $(m,n)$-von Neumann regular.
  \end{enumerate}
\end{cor}

The following result is the analog of Theorem~\ref{add3} for $(m, n)$-closed ideals.

\begin{thm} \label{add4}
  Let $R$ be a commutative ring and $m$ and $n$ positive integers.
	Then the following statements are equivalent.
  \begin{enumerate}
    \item Every proper ideal of $R$ is $(m,n)$-closed.
    \item $R$ is $(m,n)$-von Neumann regular.
  \end{enumerate}
\end{thm}

\begin{proof}
$(1) \Rightarrow (2)$ Let $x \in R$.
If $x \in U(R)$, then $x$ is $(m,n)$-vnr by Theorem~\ref{vnrfacts}(3).
If $x \notin U(R)$,
then  $I = x^{m}R$ is $(m,n)$-closed and $x^{m} \in I$.
Thus $x^{n} \in I$;
so  $x^n = x^mr$ for some $r \in R$. Hence $x$ is $(m,n)$-vnr,
and thus $R$ is $(m,n)$-von Neumann regular.

$(2) \Rightarrow (1)$  Let $I$ be a proper ideal of $R$
and $x^m \in I$ for  $x \in R$.
Since $x$ is $(m, n)$-vnr, $x^mr = x^n$ for some $r \in R$.
Thus  $x^n = x^mr \in I$; so $I$ is $(m,n)$-closed.
\end{proof}

Of course, we are mainly interested in the case when $m > n$.
The next theorem incorporates Theorem~\ref{add4} with another characterization (\cite[Theorem 2.14]{BARECENT})
of when every proper ideal is $(m,n)$-closed. Note that in Theorem~\ref{add5}(3) below, there are no conditions on $m$ other than $m > n$.

\begin{thm} \label{add5}
 Let $R$ be a commutative ring and $m$ and $n$ positive integers with $m > n$.
	Then the following statements are equivalent.
	\begin{enumerate}
    \item Every proper ideal of $R$ is $(m,n)$-closed.
    \item $R$ is $(m,n)$-von Neumann regular.
		\item dim$(R) = 0$ and $w^n = 0$ for every $w \in Nil(R)$.
  \end{enumerate}
\end{thm}

\begin{proof}
$(1) \Leftrightarrow (2)$ is Theorem~\ref{add4} and
$(1) \Leftrightarrow (3)$ is \cite[Theorem 2.14]{BARECENT}.
\end{proof}

Theorem~\ref{add5} gives a nice ring-theoretric characterization of $(m,n)$-von Neumann regular rings (for $m > n$).
This can now be used to give a characterization of strongly $\pi$-regular commutative rings
which strengthens Corollary~\ref{strongcor}.

\begin{thm}
 Let $R$ be a commutative ring. Then the following statements are equivalent.
\begin{enumerate}
\item  $R$ is strongly $\pi$-regular.
\item  There are positive integers $m$ and $n$ with $m > n$ such that $R$ is $(m,n)$-von Neumann regular.
\item  There is a positive integer $n$ such that $R$ is $(m,n)$-von Neumann regular for every positive integer $m$.
\item  dim$(R) = 0$ and there is a positive integer $n$ such that $w^n = 0$ for every $w \in Nil(R)$.
\end{enumerate}
\end{thm}

\begin{proof}
$(1) \Rightarrow (2)$ A strongly $\pi$-regular ring is $(2n,n)$-von Neuman regular for some positive integer $n$.

$(2) \Rightarrow (3)$ This follows from Corollary~\ref{strongcor}.

$(3) \Rightarrow (1)$ In particular, $R$ is $(2n,n)$-von Neumann regular, and thus strongly $\pi$-regular.

$(2) \Leftrightarrow (4)$ This is just $(2) \Leftrightarrow (3)$ of Theorem~\ref{add5}.
\end{proof}

We next investigate in more detail the pairs $(m,n)$ for which a commutative ring $R$
or an $x \in R$ is $(m,n)$-von Neumann regular.

\begin{defn}
Let $R$ be a commutative ring, $x \in R$, and $k$ a positive integer.
\begin{enumerate}
\item  $\mathcal{V}(R,x) = \{ \, (m,n) \in \mathbb{N} \times \mathbb{N} \mid x$ is $(m,n)$-vnr $\}$.
\item  $\mathcal{V}(R) = \{\, (m,n) \in \mathbb{N} \times \mathbb{N} \mid R$ is $(m,n)$-von Neumann regular $\}$.
\item  $\mathcal{B}_k = \{ \, (m,n) \in \mathbb{N} \times \mathbb{N} \mid m \leq n$ or $n \geq k \, \}$.
\item  $\mathcal{B}_{\omega} = \{ \,(m,n) \in \mathbb{N} \times \mathbb{N} \mid m \leq n \, \}$.
\end{enumerate}
\end{defn}

Then $\mathcal{V}(R) = \bigcap_{x \in R} \mathcal{V}(R,x)$ and

$$\mathbb{N} \times\mathbb{N} = \mathcal{B}_1 \supsetneq \mathcal{B}_2 \supsetneq \cdots \supsetneq \mathcal{B}_{\omega}.$$

\begin{thm}
Let $R$ be a commutative ring and $x \in R$.
\begin{enumerate}
\item  $\mathcal{V}(R,x) = \mathcal{B}_k$, where $k$ is the smallest positive integer
such that $(i,k) \in \mathcal{V}(R,x)$ for some $i > k$.
(Thus $k$ is the smallest positive integer such that $x$ is $(m,k)$-vnr for every positive integer $m$.)
If no such $k$ exists, then  $\mathcal{V}(R,x) = \mathcal{B}_{\omega}$.
\item $\mathcal{V}(R) = \mathcal{B}_k$, where $k$ is the smallest positive integer
such that $(i,k) \in \mathcal{V}(R,x)$ for some $i > k$ and every $x \in R$.
(Thus $k$ is the smallest positive integer such that $x$ is $(m,k)$-vnr for every $x \in R$ and positive integer $m$.)
If no such $k$ exists, then  $\mathcal{V}(R) = \mathcal{B}_{\omega}$.
\end{enumerate}
\end{thm}

\begin{proof}
$(1)$ follows directly from Theorem~\ref{vnrfacts}(7). Thus $(2)$ holds by definition.
\end{proof}

These ideas can also be used to classify zero-dimensional commutative rings.

\begin{defn}  Let $R$ be a commutative ring and $n$ a positive integer.
\begin{enumerate}
\item  $R$ is $n$-regular if $\mathcal{V}(R) = \mathcal{B}_n$,
i.e., $n$ is the smallest positive integer such that for every $x \in R$ and
positive integer $m$, $x^n = x^mr_m$ for some $r_m \in R$.
\item  $R$ is $\omega$-regular if for every $x \in R$, $\mathcal{V}(R,x) = \mathcal{B}_{n_x}$
for some positive integer $n_x$, but $\mathcal{V}(R) = \mathcal{B}_{\omega}$.
\end{enumerate}
\end{defn}

A commutative ring $R$ is von Neumann regular if and only if it is $1$-regular,
and $R$ is strongly $\pi$-regular if and only if it is $n$-regular for some positive integer $n$.
Note that $R$ is $\pi$-regular if and only if every $x \in R$ is
$(m,n)$-vnr for some positive integers $m$ and $n$ with $m > n$,
but a $\pi$-regular ring may be $\omega$-regular (see Example~\ref{ex7}(d)).
Thus $R$ is $\alpha$-regular for $\alpha$ a positive integer or $\omega$ if and only if $R$ is $\pi$-regular, if and only if dim$(R) = 0$.
So, in some sense, this concept measures how far a zero-dimensional
commutative ring is from being von Neumann regular.

We next give several examples. In particular, we show that if $\alpha$ is any
positive integer or $\omega$, there is a quasilocal  commutative ring $R_{\alpha}$ that is $\alpha$-regular.

\begin{exm}  \label{ex7}
Let $R$ be a commutative ring.

$(a)$  Suppose that there is an $x \in R \setminus (Z(R) \cup U(R))$ (so dim$(R) > 0$).
Then $\mathcal{V}(R) =  \mathcal{V}(R,x) = \mathcal{B}_{\omega}$ by Theorem~\ref{vnrfacts}(4).
Thus $R$ is not $\omega$-regular or $n$-regular for any positive integer $n$.

$(b)$  Suppose that $R$ is quasilocal with maximal ideal $M = (x)$
with $x^k = 0$ and $x^{k-1} \neq 0$ for an integer $k \geq 2$.
Then  $\mathcal{V}(R) = \mathcal{B}_k$ by Theorem~\ref{vnrfacts}(3)(6); so $R$ is $k$-regular.
This also holds for $k = 1$ since a field is von Neumann regular.
In particular, for a prime $p$ and any positive integer $k$, $\mathcal{V}(\mathbb{Z}_{p^k}) = \mathcal{B}_k$, and
thus $\mathbb{Z}_{p^k}$ is $k$-regular.

$(c)$  Let $R_1$ and $R_2$ be commutative rings.
Then $x = (x_1,x_2) \in R_1 \times R_2$ is $(m,n)$-vnr if and only if
$x_1$ and $x_2$ are ($m,n)$-vnr in $R_1$ and $R_2$, respectively.
Thus $\mathcal{V}(R_1 \times R_2) = \mathcal{B}_k$, where $\mathcal{V}(R_1) = \mathcal{B}_{k_1}$,
$\mathcal{V}(R_2) = \mathcal{B}_{k_2}$, and $k = $max$\{k_1,k_2\}$;
so $R_1 \times R_2$ is max$\{k_1,k_2\}$-regular
when $R_1$ and $R_2$ are $k_1$-regular and $k_2$-regular, respectively.
In particular, for distinct primes $p_1, \ldots, p_r$, positive integers $k_1, \ldots, k_r$, and $k =$ max$\{k_1, \ldots, k_r \}$,
$\mathcal{V}(\mathbb{Z}_{p_1^{k_1}} \times \cdots \times \mathbb{Z}_{p_r^{k_r}}) =
\mathcal{B}_k$, and hence $\mathbb{Z}_{p_1^{k_1}} \times \cdots \times \mathbb{Z}_{p_r^{k_r}}$ is $k$-regular.

$(d)$  Let $R = \mathbb{Z}_2[\{ X_n \}_{n \in \mathbb{N}}]/(\{ X_n^{n+1} \}_{n \in \mathbb{N}})
= \mathbb{Z}_2[\{x_n\}_{n\in \mathbb{N}}]$.
Then $R$ is a zero-dimensional quasilocal commutative ring with maximal ideal $Nil(R) = (\{x_n\}_{n\in \mathbb{N}})$;
so $R$ is $\pi$-regular.
Thus every $x \in R$ has $\mathcal{V}(R,x) = \mathcal{B}_k$ for some positive integer $k$
and $\mathcal{V}(R,x_n) = \mathcal{B}_{n+1}$ by Theorem~\ref{vnrfacts}(3)(6);
so $\mathcal{V}(R) = \mathcal{B}_{\omega}$. Hence $R$ is $\omega$-regular.
\end{exm}

\end{document}